\documentclass[12pt, twoside]{article}
\usepackage{bbm}
\usepackage{openbib}
\usepackage{amsmath,amsthm,amssymb,mathrsfs,amsfonts,mathbbold}
\usepackage{times,lineno}
\usepackage{enumerate}
\usepackage{tikz}
\usepackage{cite}
\usepackage[square,numbers,compress]{natbib}
\usepackage{pgflibraryarrows}
\usepackage{pgflibrarysnakes}							
\usepackage{hyperref}
\usepackage{microtype}
\allowdisplaybreaks

\makeatletter
\def\author#1{\gdef\autrun{\def\and{\unskip, }#1}\gdef\@author{#1}}

\makeatother

\usepackage{makeidx}
\newtheorem{theorem}{Theorem}[section]

\newtheorem{lemma}[theorem]{Lemma}

\newtheorem{proposition}[theorem]{Proposition}
\theoremstyle{definition}

\newtheorem{remark}[theorem]{Remark}

\numberwithin{equation}{section}

\allowdisplaybreaks

 \makeatletter
\definecolor{ForestGreen}{rgb}{0.15,0.416,0.18}
\definecolor{EgyptBlue}{rgb}{0.063,0.2,0.65}
\hypersetup{
	colorlinks=true,
	linkcolor=EgyptBlue,
   citecolor=EgyptBlue,
   urlcolor=ForestGreen
}

\begin{document}


\title{Symmetry and classification of positive solutions of some weighted elliptic equations}

\author{Kui Li, \ \ Mengyao Liu, \  \ Jianfeng Wu\thanks{Corresponding author. 
\newline\indent This work is supported by the Natural Science Foundation of Henan Province, China (No. 242300420232). 
\newline\indent E-mail addresses: likui@zzu.edu.cn (K. Li), 2274023811@qq.com (M. Liu), wuj@zzu.edu.cn (J. Wu).
}\\
{\small School of Mathematics and Statistics, Zhengzhou University, Zhengzhou 450001, China. } }
\date{}
\maketitle
\begin{abstract}
We study the weighted elliptic equation
\begin{equation}\label{mainequ1}
-div(|x|^{-2a}\nabla u)=|x|^{-bp}|u|^{p-2}u~ ~ \mbox{in}~ \mathbb{R}^N
\end{equation}
with $N\geq 2$, which arises from the Caffarelli-Kohn-Nirenberg inequalities. Under the assumptions of finite energy and $a_1+a_2=N-2$, 
for nonnegative solutions we prove the equivalence between equation \eqref{mainequ1} with $a=a_1$ and equation \eqref{mainequ1} with $a=a_2$. Without finite energy assumptions, for $2\leq p<2^*$ we give the optimal parameter range in which nonnegative solutions of \eqref{mainequ1} in $\mathbf{L}^\infty_{Loc}(\mathbb{R}^N)$ must be radially symmetric, and give a complete classification for these solutions in this range. 
\end{abstract}

\section{Introduction}
\quad\ \; In this paper, we study the radial symmetry and classification of solutions of the following weighted elliptic equation
\begin{equation}\label{mainequ}
-div(|x|^{-2a}\nabla u)=|x|^{-bp}|u|^{p-2}u~~~ \mbox{in}~ \mathbb{R}^N, 
\end{equation}
where $N\geq1$, $a$, $b\in \mathbb{R}$ satisfying
\begin{equation}\label{paremeter-b}
\begin{aligned}
b\in
\begin{cases}~[a, a+1], ~ &\mbox{if}~ N\geq3, \\
  (a, a+1], ~&\mbox{if}~ N=2, \\
  (a+\frac{1}{2}, a+1], ~ &\mbox{if}~ N=1
\end{cases}
\end{aligned}
\end{equation}
and
\begin{equation}\label{paremeter-p}
p=\frac{2N}{N-2+2(b-a)}. 
\end{equation}

Note that if $a=b=0$, then 
\begin{equation*}
\begin{aligned}
p=
\begin{cases}~ \frac{2N+2l}{N-2}, ~~~ &\mbox{if}~ N\geq3, \\~ \infty, ~~~ &\mbox{if}~ N=1, 2
\end{cases}
\end{aligned}
\end{equation*}
is the Sobolev exponent which is denoted by $2^*$. Hence equation \eqref{mainequ} is a natural extension of Lane-Emden equation
\begin{equation}\label{L-E}
 -\Delta u=|u|^{q-2}u, ~~~ x\in \mathbb{R}^N
\end{equation}
with $q>1$, and is also a natural extension of the weighted Lane-Emden equation (also called H\'{e}non- Lane-Emden equation)
\begin{equation}\label{weightedL-E}
 -\Delta u=|x|^l|u|^{q-2}u, ~ ~ ~ x\in \mathbb{R}^N
\end{equation}
with $q>1$ and $l\in \mathbb{R}$. For equation \eqref{weightedL-E}, we introduce the Hardy-Sobolev exponent
\begin{equation*}
\begin{aligned}
2^*(l)=
\begin{cases}~ \frac{2N+2l}{N-2}, ~~~ &\mbox{if}~N\geq3, \\
  ~ \infty, ~~~ &\mbox{if}~N=1, 2. 
\end{cases}
\end{aligned}
\end{equation*}

Let
\begin{equation}\label{paremeter-a}
a_c:=\frac{N-2}{2}. 
\end{equation}
Under the conditions $a<a_c$, \eqref{paremeter-b} and \eqref{paremeter-p}, equation \eqref{mainequ} arises from the Caffarelli-Kohn-Nirenberg inequalities (see \cite{CKN} and \cite{CW})
\begin{equation}\label{CCNinequality}
\big(\int_{\mathbb{R}^N}\frac{|u|^p}{|x|^{bp}}\mathbf{d}x\big)^{\frac{2}{p}}\leq \textbf{C}_{a, b, N}\int_{\mathbb{R}^N}\frac{|\nabla u|^2}{|x|^{2a}}\mathbf{d}x, ~~~~ \forall u\in\mathcal{D}_{a, b}, 
\end{equation}
which are the interpolations between the usual Sobolev
inequality corresponding
to $a=b=0$ and the weighted Hardy inequalities corresponding
to $b=a+1$. Optimal functions for inequality \eqref{CCNinequality} solve equation \eqref{mainequ}. The exponent $p$ given in \eqref{paremeter-p} is determined by the invariance of the inequality and the space $\mathcal{D}_{a, b}$ under the scaling $u_R(x):=R^{\frac{N-2-2a}{2}}u(Rx)$ for any $u\in \mathcal{D}_{a, b}$ and $R>0$, $\textbf{C}_{a, b, N}$ denotes the optimal constant in \eqref{CCNinequality} which depends only on the parameters $a, b$ given above and the dimension $N$, and the space $\mathcal{D}_{a, b}$ is defined by
\begin{equation}\label{D-space}
\mathcal{D}_{a, b}:=\{u\in \mathbf{L}^p(\mathbb{R}^N, |x|^{-b}\mathbf{d}x): \nabla u\in \mathbf{L}^2(\mathbb{R}^N, |x|^{-a}\mathbf{d}x)\}
\end{equation}
with norm $\|u\|^2=\big\||x|^{-a}\nabla u\big\|_{\mathbf{L}^2(\mathbb{R}^N}^2+\big\||x|^{-b} u\big\|_{\mathbf{L}^p(\mathbb{R}^N}^2$ for $u\in\mathcal{D}_{a, b}$. \par

Both symmetry and classification of the solutions of equation \eqref{mainequ}, equation \eqref{L-E} and equation \eqref{weightedL-E} are central themes in mathematics and the physical sciences. \par

A classical result for equation \eqref{L-E} is the following classification, which was established by Gidas and Spruck in their celebrated article \cite{GS}: \\
\textbf{Theorem A.} \emph{Assume that $1<q<2^*-1$. If $u$ is a nonnegative solution of equation \eqref{L-E}, then $u\equiv 0$. }\\
Later, Chen and Li \cite{CL1} gave a new proof for the above theorem by  Kelvin transform and the method of moving planes, and also studied  the critical case:  \\
\textbf{Theorem B. } \emph{Assume that $N\geq 3$, $q=2^*-1$. If $u$ is a nonnegative solutions of equation \eqref{L-E}, then $u$ must be radially symmetric, and hence either $u\equiv 0$ or $u$ takes the form
\begin{equation}\label{form-LE}
\Bigg(\frac{\sqrt{N(N-2)}\lambda}{\lambda^2+|x-x_0|^2}\Bigg)^{(N-2)/2}
\end{equation}
for some $\lambda>0$ and $x_0\in\mathbb{R}^N$. } \par

For equation \eqref{weightedL-E}, it has been proved that there is no positive solution
$u\in C^2(\mathbb{R}^N\backslash\{0\})\cap C(\mathbb{R}^N)$
for $l\leq -2$ (see \cite{DP}). If $l>-2$, Phan and Souplet \cite{PS} conjectured that\\
\textbf{Conjecture C. } \emph{Let $N\geq2$ and $l>-2$. If $0<q<2^*(l)-1$, then the nonnegative solutions of equation \eqref{weightedL-E} must be zero. }\\
For $N=3$, Phan and Souplet \cite{PS} showed that this conjecture holds true for bounded solutions, and Li and Zhang \cite{LZ1} proved the same results without the boundness assumptions. For $N\geq3$ and $l\geq 0$, Reichel and Zou \cite{RZ} gave a positive answer to this conjecture. And this conjecture was completely solved by Guo and Wan \cite{GW} (also see \cite{LZ2}). \par

Suppose that $a<a_c$, \eqref{paremeter-b}, and \eqref{paremeter-p} hold. Symmetry and classification results of the solutions of \eqref{mainequ} have also been obtained for various values of $a$ and $b$. For $N\geq2$ and $0\leq a<a_c$, it has been proved in \cite{CC}, \cite{DELT} and \cite{H} that nonnegative solutions that are bounded near the origin are radially symmetric (if $a=b=0$, then also up to a translation), and are equal to $u^*$ up to a scaling (if $a=b=0$, then also up to a translation), where
\begin{equation}\label{form-CKN}
u^*(x)=C(a, b, N)\Bigg(1+|x|^{(p-1)(a_c-a)}\Bigg)^{-2/(p-2)}
\end{equation}
and $C(a, b, N)$ is explicit which is depending only on $a$, $b$ and $N$. The case $a<0$ is by far more complicated. Under the finite energy assumption, Dolbeault, Esteban and Loss \cite{DEL1} showed that positive solutions are radial symmetric and equal to $u^*$ up to a scaling for $N\geq2$, $a<0$ and
\begin{equation}\label{b-direct}
\frac{N(N-1)+4N(a-a_c)^2}{6(N-1)+8(a-a_c)^2}+a-a_c\leq b<a+1. 
\end{equation}
For $N\geq 3$, $a<0$ and
\begin{equation}\label{b-fs}
a<b<b_{FS}(a):=\frac{N(a-a_c)}{2\sqrt{(a-a_c)^2+N-1}}+a-a_c, 
\end{equation}
Felli and Schneider \cite{FS} proved that optimal functions for inequality \eqref{CCNinequality} are non-radial, Lin and Wang \cite{LW} showed that these optimal functions have exact $O(N-1)$ symmetry, and hence equation \eqref{mainequ} has non-radial positive solutions. It was a longstanding conjecture that the Felli-Schneider curve $b_{FS}(a)$ is the threshold between the symmetry and the symmetry breaking region. Finally, Dolbeault, Esteban and Loss \cite{DEL2} gave an affirmative answer to this conjecture and showed that if $N\geq 2$, $a<b<a+1$, and either $a\in (0, a_c]$ and $b>0$ or $a<0$ and $b\geq
b_{FS}(a)$, then positive solutions of \eqref{mainequ} are radially symmetric and equal to $u^*$ up to a scaling with $u^*$ is given \eqref{form-CKN}. \par

The aim of this paper is to extend the above results. We first show that equation \eqref{mainequ} for $a\geq a_c$ is equivalent to equation \eqref{mainequ} for $a\leq a_c$. 
\begin{theorem}\label{transformtheorem}
Assume that $ N\geq 2$, $a_i$, $b_i\in\mathbb{R}$ satisfy \eqref{paremeter-b} and $p_i$ satisfy \eqref{paremeter-p} for $i=1, 2$. If $a_1+a_2=2a_c$, $b_2-a_2=b_1-a_1$, $u_1$ and $u_2$ are two nonnegative measurable functions in $\mathbb{R}^N$ satisfying $|x|^{-a_1}u_1=|x|^{-a_2}u_2$, then $u_1$ is a solution of equation \eqref{mainequ} in $\mathcal{D}_{a, b}$ with $a=a_1$ and $b=b_1$ if and only if $u_2$ is a solution of equation \eqref{mainequ} in $\mathcal{D}_{a, b}$ with $a=a_2$ and $b=b_2$. 
\end{theorem}

By Theorem \ref{transformtheorem}, we only need to consider the case of $a\leq a_c$ when studying the symmetry and classification of nonnegative solutions of equation \eqref{mainequ}. Without finite energy assumption, we have the following results. \par

\begin{theorem}\label{maintheorem}
Suppose that $ N\geq 2$, $p$ satisfies \eqref{paremeter-p}, $a=a_c$ with $b$ satisfying \eqref{paremeter-b} or $a<a_c$ with $b=a+1$ or $a\in[0, a_c)$ with $a\leq b< a+1$ or $a<0$ with $ b_{FS}(a)\leq b< a+1$. Then nonnegative solutions of equation \eqref{mainequ} in $\mathbf{L}^\infty_{Loc}(\mathbb{R}^N)$ must be radially symmetric (if $a=b=0$, then and up to a translation). \par

Moreover, let $u\in\mathbf{L}^\infty_{Loc}(\mathbb{R}^N)$ be a nonnegative solution of equation \eqref{mainequ}. If $a=a_c$ with $b$ satisfying \eqref{paremeter-b} or $a<a_c$ with $b=a+1$, then $u\equiv0$; if $a\in[0, a_c)$ with $a\leq b< a+1$ or $a<0$ with $ b_{FS(a)}\leq b< a+1$, then either $u\equiv0$ in $\mathbb{R}^N$ or $u$ is equal to $u^* $ up to a scaling  (and also up to a translation for $a=b=0$), where $u^*$ is given \eqref{form-CKN}. 
\end{theorem}

\begin{remark}
Assume that $b=a<0$. An unbounded sequence of non-radial positive solutions of \eqref{mainequ} in $\mathcal{D}_{a, b}$ have been built in \cite{MW} for $N\geq 5$, and the existence of at least one non-radial positive solution in $\mathcal{D}_{a, b}$ has also been obtained in \cite{BCG}, \cite{DGG} and \cite{T} for $N\geq 3$. 
\end{remark}

\begin{remark}
According to the classical work of Felli and Schneider \cite{FS}, we give the optimal parameter range in which nonnegative solutions of \eqref{mainequ1} in $\mathbf{L}^\infty_{Loc}(\mathbb{R}^N)$ must be radially symmetric. 
\end{remark}

The proof of Theorem \ref{transformtheorem} and Theorem \ref{maintheorem} can be outlined as follows. By direct calculations, we conclude that to prove Theorem \ref{transformtheorem}, we only need to show that if $u_1\in\mathcal{D}_{a_1, b_1}$, then $u_2\in\mathcal{D}_{a_2, b_2}$. To do this, we show that $\bar{u}$ is monotonic (see Lemma \ref{lem-monotonic}) and has some kind of decay estimates (see Lemma \ref{lem-decay}). And thus we prove a Hardy type inequality (see Lemma \ref{lem-hardy}). With the help of this inequality, we show $u_2\in\mathcal{D}_{a_2, b_2}$ and prove Theorem \ref{transformtheorem}. \par

The proofs of Theorem \ref{maintheorem} are divided into three cases. For the case of $a=a_c$ and the case of $a<a_c$ with $b=a+1$, we use spherical average and ODE theory to establish two Liouville-type results (see Proposition \ref{liouville-type-less} and Proposition \ref{liouville-type-eigenvalue}). In the case of $a\in[0,a_c)$ with $a\leq b< a+1$ or $a<0$ with $ b_{FS(a)}\leq b< a+1$, we make use of the moving planes methods and stability theory to show that $u$ has polynomial decay (see Theorem \ref{decaytheorem}), and prove that $u$ has finite energy (see Lemma \ref{finiteenergylemma}). With these preparations, we apply Theorem 1.2 in \cite{DEL2} to obtain Theorem \ref{maintheorem}. \par

This paper is organized as follows: in section 2, we prove Theorem \ref{transformtheorem}; in section 3, we give some Liouville-type results; in section 4, we  study decay estimates for solutions; and finally in section 4, we prove Theorem \ref{maintheorem}.

\section{Equivalence of equations}
\numberwithin{equation}{section}
 \setcounter{equation}{0}
In this section, we prove Theorem \ref{transformtheorem}. By elliptic theory, it is easy to show that the solutions of equation \eqref{mainequ} in $\mathbf{L}^\infty_{Loc}(\mathbb{R}^N)$ or in $\mathcal{D}_{a,b}$ is smooth in $\mathbb{R}^N\backslash\{0\}$. \par

We first give some notations that will be used in the rest of this paper. For $0\neq x\in \mathbb{R}^N $, let $\theta=\frac{x}{|x|} \in S^{N-1}$ and $r=|x|$, and for a continuous function $u(x)$ in $\mathbb{R}^N\backslash\{0\}$ with $N\geq 2$, we write $u=u(r,\theta)$ and define
$$\overline{u}(r)=\frac{1}{\omega_N}\int_{S^{N-1}}u(r,\theta)d\theta, $$
where $\omega_N=|S^{N-1}|$ and $d\theta$ is the surface measure on $S^{N-1}$. \par

\begin{lemma}\label{lem-monotonic}
Assume that $ N\geq 2$, $a$, $b\in\mathbb{R}$ satisfy \eqref{paremeter-b} and $p$ satisfies \eqref{paremeter-p}. If $u\in\mathcal{D}_{a, b}$ is a nonnegative solution of equation \eqref{mainequ}, then $\bar{u}$ is monotonic in $(0, \infty)$. 
\end{lemma}
\begin{proof}
Let $N'=N-2a$ and $\tau=-bp+2a$. Then we have 
\begin{equation}\label{lem-monotonic1}
-\bar{u}''-\frac{N'-1}{r}\bar{u}'=r^\tau \overline{u^{p-1}}~\mbox{in}~(0,\infty). 
\end{equation}
Hence $r^{N'-1}\bar{u}'$ is decreasing in $(0, \infty)$. \par

If $N'\geq2$, then by the proofs of Lemma 4.1 in \cite{LZ3}, $\lim\limits_{r\rightarrow 0^+}r^{N'-1}\bar{u}'(r)\leq0$ and $\bar{u}$ is decreasing in $(0, \infty)$ and. If $N'<2$, then also by the proofs of Lemma 4.1 in \cite{LZ3}, $\lim\limits_{r\rightarrow \infty}r^{N'-1}\bar{u}'(r)\geq 0$, and $\bar{u}$ is increasing in $(0, \infty)$. In both cases, we find that $\bar{u}$ is monotonic. 
\end{proof}

Applying Lemma \ref{lem-monotonic}, we can obtain some decay estimates for $u$. 
\begin{lemma}\label{lem-decay}
If $p>2$, then there exists a positive constant $C = C(N,a,b)$ such that
\begin{equation}
\overline{u} (r) \leqslant C r^{-\frac{N-2a-2}{2}}
\end{equation}
for any $r>0$. 
\end{lemma}

\begin{proof} We will choose a proper test-function and rescale the solutions (see e.g. \cite{LZ4} or \cite{MP}). 
Let $\psi(x)\in C_c^\infty(\mathbb{R}^N)$ be a cut-off function: 
\begin{equation}\label{lem-decay1}
\psi(x)=
\begin{cases} 1, & |x|<2, \\
  0, & |x|>3
\end{cases}
\end{equation}
with $0\leq\psi(x)\leq1$, $|div(|x|^{-2a}\nabla \psi)|\leq C_1\psi^{\frac{1}{p-1}}$ with $C=C(N, a, b)>0$. \par

By the properties of $\psi(x)$, we deduce that
\begin{equation}\label{lem-decay2}
\begin{aligned}
\int_{\mathbb{R}^N} |x|^{-bp} u^{p-1} \psi&=-\int_{\mathbb{R}^N} div(|x|^{-2a}\nabla u)\psi \\
&=-\int_{\mathbb{R}^N}div(|x|^{-2a}\nabla \psi)u\leq C_1\int_{2\leq|x|\leq3} u\psi^{\frac{1}{p-1}}. 
\end{aligned}
\end{equation}
Hence we conclude
\begin{equation}\label{lem-decay3}
\int_{\mathbb{R}^N} |x|^{-bp} u^{p-1} \psi\leq C_2(\int_{2\leq|x|\leq3} u^{p-1}\psi)^{\frac{1}{p-1}}\leq C_3(\int_{\mathbb{R}^N} |x|^{-bp} u^{p-1}\psi)^{\frac{1}{p-1}}
\end{equation}
and
\begin{equation}\label{lem-decay4}
 \int_{B_2} |x|^a u^{p-1}\leq\int_{\mathbb{R}^N} |x|^a u^{p-1} \psi\leq C_4, ~~~ \omega_N\int_{1/2}^2r^{N-1}\bar{u}\mathbf{d}r=\int_{B_{2}\backslash B_{1/2}} u\leq C_5, 
\end{equation}
where $\omega_N$ is the volume of $B_1$ and $C=C(N, a, b)>0$. By Lemma \ref{lem-monotonic}, we deduce that
\begin{equation}\label{lem-decay5}
\overline{u} (1) \leqslant C
\end{equation}
with $C=C(N, a, b)$. 

Since $R^{\frac{N-2a-2}{2}} u(Rx)$ is also a solution of equation \eqref{mainequ}, replacing $u$ by$R^{\frac{N-2a-2}{2} u(Rx)}$ in \eqref{lem-decay5} and changing variables we prove this lemma. 
\end{proof}

Using Lemma \ref{lem-decay}, we have the following Hardy type inequality. 
\begin{lemma}\label{lem-hardy}
If $p>1$, $a\neq a_c$ and $u\in\mathcal{D}_{a, b}$ is solution of equation \eqref{mainequ}, then there exists a positive constant $C = C(N, a, b, p)$ such that
\begin{equation}\label{lem-hardy1}
\int_{\mathbb{R}^N}|x|^{-2a-2}u^2\leq C\big(\int_{\mathbb{R}^N}|x|^{-2a}|\nabla u|^2+1\big). 
\end{equation}
\end{lemma}
\begin{proof} In the following proof, $C_i$ are positive constants which depend only on $a$ and $N$. \par

Let $\eta \in C_c^\infty (\mathbb{R}^N)$ and $\eta_j =  \eta (jx)$ with $0 \leqslant \eta \leqslant 1$, $\eta = 0$ in $B_1$ and $\eta = 1$ in $B_2^c$.  Then $\eta_j$ is zero near the origin. \par

If $u\in C_c^\infty(\mathbb{R}^N)$, then by  Hardy inequality we have
\begin{equation}\label{lem-hardy2}
\begin{aligned}
\int_{\mathbb{R}^N}|x|^{-2a-2}u^2\eta_j^2&\leq C_6\int_{\mathbb{R}^N}|x|^{-2a}|\nabla (u\eta_j)|^2\\
&=C_6\int_{\mathbb{R}^N}|x|^{-2a}\big(|\nabla u|^2\eta^2_j+u^2|\nabla \eta_j|^2+2u\eta_j\nabla u\nabla \eta\big). 
\end{aligned}
\end{equation}
Therefore, H$\ddot{\mathrm {o}}$lder's inequality implies that
\begin{equation}\label{lem-hardy3}
\int_{\mathbb{R}^N}|x|^{-2a-2}u^2\eta_j^2\leq C_7\big(\int_{\mathbb{R}^N}|x|^{-2a}\big(|\nabla u|^2\eta^2_j+u^2|\nabla \eta|^2_j\big). 
\end{equation}
Note that
\begin{equation}\label{lem-hardy4}
\begin{aligned}
\int_{\mathbb{R}^N}|x|^{-2a}u^2|\nabla \eta_j|^2&\leq C_8 \int_{B_{\frac{2}{j}}\backslash B_{\frac{1}{j}}} |x|^{-2a-2} u^2\\
&\leq 2C_8 \int_{B_{\frac{2}{j}}\backslash B_{\frac{1}{j}}}|x|^{-2a-2} \big((u-\bar{u})^2+|x|^{-2a-2} \bar{u}^2\big)\\
&=2\omega_N C_8\int^{\frac{2}{j}}_{\frac{1}{j}} r^{N-2a-3}\big(\overline{(u(r, \theta)-\bar{u}(r))^2}+\bar{u}^2\big)\mathbf{d}r, 
\end{aligned}
\end{equation}
hence Lemma \ref{lem-decay} 
yields that
\begin{equation}\label{lem-hardy5}
\begin{aligned}
\int_{\mathbb{R}^N}|x|^{-2a}u^2|\nabla \eta_j|^2&\leq \frac{2\omega_N C_8}{N-1}\int^{\frac{2}{j}}_{\frac{1}{j}} r^{N-2a-1}\overline{|\nabla u|^2}\mathbf{d}r+C_9\int^{\frac{2}{j}}_{\frac{1}{j}} r^{-1}\mathbf{d}r\\
&=\frac{2\omega_N C_8}{N-1}\int_{B_{\frac{2}{j}}\backslash B_{\frac{1}{j}}}|x|^{-2a}|\nabla u|^2+C_9\ln 2. 
\end{aligned}
\end{equation}
By \eqref{lem-hardy3}, \eqref{lem-hardy5} and letting $j\rightarrow\infty$, we conclude that this lemma holds for $u\in C_c^\infty(\mathbb{R}^N)$ with $C=\max\{C_7, C_7C_9\ln 2\}$. \par

Now suppose that $u\not\in C_c^\infty(\mathbb{R}^N)$. We choose a test function $\varphi$ and repeat the above arguments with $\eta_j$ replacing by $\phi_j:=\varphi(\frac{x}{j})$. Then $\frac{1}{j}$ in \eqref{lem-hardy4} and \eqref{lem-hardy5} will be changed to $j$. Letting $j\rightarrow\infty$, we also obtain this lemma.
\end{proof}

Now, we turn to the proof of Theorem \ref{transformtheorem}. \par

\emph{Proof~ of~ Theorem \ref{transformtheorem}: } Without loss of generality, we assume that $a_1\neq a_2$, $u_1\in \mathcal{D}_{a_1,b_1}$ is a solution of the equation \eqref{mainequ} and $|x|^{-a_1}u_2=|x|^{a_2-a_1}u_1$. Then by calculations, we deduce that $u_2$ is a solution of equation \eqref{mainequ} with $a=a_2$ and $b=b_2$. Hence, we only need to check that $u_2\in \mathcal{D}_{a_2,b_2}$. \par

Note that
\begin{equation}\label{transformtheorem1}
\begin{aligned}
\int_{\mathbb{R}^N}|x|^{-b_2p}u_2^p&=\int_{\mathbb{R}^N}|x|^{-b_2p+(a_2-a_1)p}u_2^p\\
&=\int_{\mathbb{R}^N}|x|^{-b_2p+(b_2-b_1)p}u_2^p=\int_{\mathbb{R}^N}|x|^{-b_1p}u_1^p
\end{aligned}
\end{equation}
and
\begin{equation}\label{transformtheorem2}
\begin{aligned}
\int_{\mathbb{R}^N}|x|^{-2a_2}|\nabla u_2|^2&=\int_{\mathbb{R}^N}|x|^{-2a_2}|\nabla (|x|^{a_2-a_1}u_1)|^2\\
&\leq 2\int_{\mathbb{R}^N}|x|^{-2a_1}|\nabla u_1|^2+2(a_1-a_2)^2\int_{\mathbb{R}^N}|x|^{-2a_1-2}u_1^2, 
\end{aligned}
\end{equation}
hence Lemma \ref{lem-hardy} and the fact that  $u_1\in \mathcal{D}_{a_1, b_1}$ imply that  $u_2\in \mathcal{D}_{a_2, b_2}$.

\section{Liouville-type results}
\numberwithin{equation}{section}
 \setcounter{equation}{0}
\quad\ \; In this section, we give some Liouville-type results for equation \eqref{mainequ}. Our first Liouville-type result is

\begin{proposition}\label{liouville-type-less}
Assume that $ N\geq 2$, $a=a_c$ and  $b$ satisfies \eqref{paremeter-b}. If $u\in\mathbf{L}^\infty_{Loc}(\mathbb{R}^N)$ is a nonnegative solution of equation \eqref{mainequ}, then $u\equiv0$ in $\mathbb{R}^N$. 
\end{proposition}
\begin{proof}
By Jessen's inequality we have
\begin{equation}\label{liouville-type-less1}
-\bar{u}''-\frac{1}{r}\bar{u}'\geq \frac{\bar{u}^{p-1}}{r^2}~ \mbox{in}~ (0, \infty). 
\end{equation}
For $t\in \mathbb{R}$, we define $r=e^t$ and $y(t)=\bar{u}(r)$. Then direct calculations imply that
\begin{equation}\label{liouville-type-less2}
y_{tt}+y^{p-1}\leq0~ \mbox{in}~ \mathbb{R}. 
\end{equation}
Hence $y$ is convex and nonnegative in $\mathbb{R}$, which implies that $y\equiv 0$ in $\mathbb{R}$ and $u\equiv0$ in $\mathbb{R}^N$. 
\end{proof}

For $a<a_c$, we have
\begin{proposition}\label{liouville-type-eigenvalue}
Suppose that $ N\geq 2$, $a< a_c$ and $b=a+1$. If $u\in\mathbf{L}^\infty_{Loc}(\mathbb{R}^N)$ is a nonnegative solution of equation \eqref{mainequ}, then $u\equiv0$ in $\mathbb{R}^N$. 
\end{proposition}

The proof of Proposition \ref{liouville-type-eigenvalue} can be obtained by combination of Lemma 2.2 and Lemma 2.5 in \cite{CLZ}. Here we give another simple proof. \par

\emph{Proof~ of~ Proposition~ \ref{liouville-type-eigenvalue}: }
By direct calculations, we have
\begin{equation}\label{liouville-type1}
-\bar{u}''-\frac{N'-1}{r}\bar{u}'=\frac{\bar{u}}{r^2}~ \mbox{in}~ (0, \infty)
\end{equation}
with $N'=N-2a$ and $\tau=-2$. Let $z(t)=\overline{u}(e^t)$ for $t\in \mathbb{R}$. Then $y$ is bounded near $-\infty$ and satisfies
\begin{equation}\label{liouville-type2}
z_{tt}+(N'-2)z_t+z=0~\mbox{in}~ \mathbb{R}. 
\end{equation}
When $a\in (a_c, N/2)$, it is well known that the linear equation \ref{liouville-type2} does not have positive solutions that are bounded near $-\infty$. Thus, by the maximum principle, we obtain this proposition.

\section{Decay estimates for positive solutions}
\numberwithin{equation}{section}
 \setcounter{equation}{0}
\quad\ \; We will use the moving planes methods to study the decay estimates of locally bounded and nonnegative solution of equation \eqref{mainequ}. 
\begin{theorem}\label{decaytheorem}
Suppose that $N\geq 2$, $a, b\in \mathbb{R}$ with $a<a_c$  and \eqref{paremeter-b} hold. If $u\in\mathbf{L}^\infty_{Loc}(\mathbb{R}^N)$ is a nonnegative solution of equation \eqref{mainequ}, then 
\begin{equation}\label{decaytheorem1}
|u|\leq C_0|x|^{-(N-2a-2)}~ \mbox{for}~ |x|\gg1, 
\end{equation}
where $C_0$ is a nonnegative constant which may depend on $u$. 
\end{theorem}

\begin{remark}\label{decayremark}
By Theorem \ref{decaytheorem}, all radial and non-radial nonnegative solutions of equation \eqref{mainequ} in $\mathbf{L}^\infty_{Loc}(\mathbb{R}^N)$ have decay estimates. 
\end{remark}

For $x\in \mathbb{R}^N\backslash\{0\}$, let $r=|x|$, $t=\ln r$ and $w(t, \theta)=r^{\frac{N-2a-2}{2}} u(r, \theta)$. Then $w$ solves the equation
\begin{equation}\label{transforequation}
\left\{ 
\begin{aligned}
w_{tt}&+\Delta_{S^{N-1}}w-\frac{(N-2a-2)^2}{4}w+w^{p-1}=0~~ \mbox{in}~ \mathbb{R}\times S^{N-1}, \\
&\lim\limits_{t\rightarrow-\infty}w(t, \theta)=0~ \mbox{uniformly~ for~ } \theta\in S^{N-1}. 
\end{aligned}\right. 
\end{equation}  

For any $\lambda \in \mathbb{R}$, we define
\begin{equation*}
\begin{aligned}
&\Sigma_\lambda=\{(t, \theta): t<\lambda, ~ \theta\in S^{N-1}\}, \\
&w^\lambda=w(2\lambda-t, \theta)-w(t, \theta), \\
&T_\lambda=\partial\Sigma_\lambda=\{(t, \theta): t=\lambda, ~ \theta\in S^{N-1}\}. \\
\end{aligned}
\end{equation*}
Then for any $\lambda \in \mathbb{R}$, 
\begin{equation}\label{transforequation1}
w^\lambda=0, ~~ \forall (t, \theta)\in T_\lambda. 
\end{equation}

By the limit in \eqref{transforequation}, we have
\begin{equation}\label{transforequation2}
\lim \limits_{t\rightarrow -\infty} \inf \limits_{\omega\in S^{N-1}} w^\lambda(t,\omega)\geq 0, ~~~ \forall~ \lambda \in \mathbb{R}. 
\end{equation}
Besides, direct calculations imply that $w^\lambda$ solves the equation
\begin{equation}\label{transforequation3}
\begin{aligned}
w^\lambda_{tt}+\Delta_{S^{N-1}} w^\lambda-\frac{(N-2a-2)^2}{4}w^\lambda+ c^\lambda w^\lambda=0, ~~ (t, \theta)\in\Sigma_\lambda, 
\end{aligned}
\end{equation}
where $c^\lambda(t, \theta)=(p-1)\int_0^1\big(sw(2\lambda-t, \theta)+(1-s)w(t, \theta)\big)^{p-2}\mathbf{d}s>0$. \par

If $w^\lambda<0$, then $c^\lambda(t, \theta)<(p-1)w^{p-2}(t, \theta)$. Hence there exists $\lambda_0\in\mathbb{R}$ such that for any $\lambda\in\mathbb{R}$, $\lambda_1\leq \lambda_0$, we have
\begin{equation}\label{transforequation4}
c^\lambda(t, \theta)-\frac{(N-2a-2)^2}{8}<0, ~~ \forall(t, \theta)\in \Sigma_{\lambda_1}\cap\Omega_{w^\lambda}^-
\end{equation}
with $\Omega_{w^\lambda}^-=\{(t, \theta): w^\lambda(t, \theta)<0\}$, and $w^\lambda$ solves the equation
\begin{equation}\label{transforequation5}
v_{tt}+\Delta_{S^{N-1}} v-\frac{(N-2a-2)^2}{8} v\leq0~\mbox{in}~\Sigma_{\lambda_2}\cap\Omega_v^-, 
\end{equation}
where $\lambda_2=\min\{\lambda, \lambda_1\}$. \par

By equation \eqref{transforequation}, \eqref{transforequation3} and  \eqref{transforequation5}, we can use the moving planes methods prove

\begin{lemma}\label{twocases}
Either $w_t(t, \theta)>0$ for all $(t, \theta)\in \mathbb{R}\times S^{N-1}$ or there exists $\Lambda \in \mathbb{R}$ such that 
$w(2\Lambda-t, \omega)=w(t, \omega)$ for all $(t, \theta)\in \mathbb{R}\times S^{N-1}$. 
\end{lemma}

Now, we prove Theorem \ref{decaytheorem}. \par
\emph{Proof~ of~ Theorem \ref{decaytheorem}: } We first consider the first case obtained in Lemma \ref{twocases}. If $w_t(t, \theta)>0$ for all $(t, \theta)\in \mathbb{R}\times S^{N-1}$, then $\zeta(x):=\frac{N-2a-2}{2}u+\nabla u\cdot x=r^{-\frac{N-2a-2}{2}}w_t>0$ for any $x\in\mathbb{R}^N\backslash\{0\}$. Hence $\zeta(x)$ is a positive solution of
\begin{equation}\label{eigenequ}
-div(|x|^{-2a}\nabla \zeta)=(p-1)|x|^{-bp}u^{p-2}\zeta~~~ \mbox{in}~ \mathbb{R}^N. 
\end{equation}
By Lemma 3.4 in \cite{LZ2}, we conclude that $u$ is stable solution of equation \eqref{mainequ}. Then by Theorem 1.1 in \cite{DG}, we deduce that $u=0$ in $\mathbb{R}^N$ and Theorem \ref{decaytheorem} holds. \par

If $w(2\Lambda-t, \theta)=w(t, \theta)$ for all $(t, \theta)\in \mathbb{R}\times S^{N-1}$, then we have
\begin{equation}\label{decaytheorem2}
|x|^{\frac{N-2a-2}{2}} u(x)=e^{(N-2a-2)\Lambda}|x|^{-\frac{N-2a-2}{2}}u(\frac{e^{2\Lambda}x}{|x|^2}), ~~ \forall x\in \mathbb{R}^N\backslash\{0\}. 
\end{equation}
Hence
\begin{equation}\label{decaytheorem3}
|u|\leq C_0|x|^{-(N-2a-2)}~ \mbox{for}~ |x|\gg1, 
\end{equation}
where $C_0=e^{(N-2a-2)\Lambda}\sup\limits_{|x|\leq e^{2\Lambda}}u(x)\in [0,\infty)$.

\section{Classification of solutions}
\numberwithin{equation}{section}
 \setcounter{equation}{0}
\quad\ \; In this section, we give the proof of Theorem \ref{maintheorem}. If $a<a_c$ and \eqref{paremeter-b} holds, then $N-2a>0$, $N-2a-2-2(N-2a-2)<0$ and $N-bp-p(N-2a-2)<0$. Hence by Theorem \ref{decaytheorem}, we have
\begin{lemma}\label{finiteenergylemma}
Assume that $N\geq 2$, $a, b\in \mathbb{R}$ with $a<a_c$  and \eqref{paremeter-b} holding. If $u\in\mathbf{L}^\infty_{Loc}(\mathbb{R}^N)$ is a nonnegative solution of equation \eqref{mainequ}, then $u\in\mathcal{D}_{a, b}$. 
\end{lemma}

With the help of Lemma \ref{finiteenergylemma}, we now can prove Theorem \ref{maintheorem}. \par

\emph{Proof~ of~ Theorem \ref{maintheorem}: } If $a=a_c$ with $b$ satisfying \eqref{paremeter-b}, then by Proposition \ref{liouville-type-less}, we obtain this theorem. If $a<a_c$ with $b=a+1$, then by Proposition \ref{liouville-type-eigenvalue}, we also get this theorem. If $a\in[0, a_c)$ with $a\leq b< a+1$ or $a<0$ with $ b_{FS}(a)\leq b< a+1$, then by Lemme \ref{finiteenergylemma}, we know that $u\in\mathcal{D}_{a, b}$. Hence by Theorem 1.2 in \cite{DEL2}, Theorem \ref{maintheorem} also holds

\section*{Acknowledgements}
\noindent

The authors thank the referees for their valuable suggestions.

\end{document}